\documentclass[12pt]{elsarticle}
\usepackage[english]{babel}
\usepackage{tikz}

\usepackage{amsfonts,amsxtra,amsthm}
 \newcommand\dom{\operatorname{dom}}

\DeclareMathOperator{\re}{Re}

\DeclareMathOperator{\Span}{span}
 \DeclareMathOperator{\ran}{ran}  
 
 \DeclareMathOperator{\diag}{diag}
 \DeclareMathOperator{\logg}{Log}
 
  \DeclareMathOperator{\opt}{\emph{\textbf{t}}}

 \newcommand{\bb}[1]{\mathbf{#1}}
\newcommand{\EE}{\mathcal E} 
 
\newtheorem{theorem}{Theorem}[section]
 
 \newtheorem{lemma}[theorem]{Lemma}
 \newtheorem{proposition}[theorem]{Proposition}
 \theoremstyle{definition}
 \newtheorem{definition}[theorem]{Definition}
 \theoremstyle{remark}
 \newtheorem{remark}[theorem]{Remark}
 
 \numberwithin{equation}{section}

\renewcommand{\log}{\textrm{Log}}

\begin{document}

\title{On the standing waves of the NLS-log equation  with point interaction on a star graph}

\author{Nataliia Goloshchapova}
\ead{nataliia@ime.usp.br}


\address{University of S\~{a}o Paulo, Rua do Mat\~{a}o 1010,  CEP 05508-090, S\~{a}o Paulo, SP, Brazil}

\begin{abstract}
We study  the nonlinear Schr\"odinger equation with logarithmic nonlinearity on a star graph $\mathcal{G}$. At the vertex an interaction occurs described by a boundary condition
of delta type with strength $\alpha\in \mathbb{R}$.
We investigate   orbital stability  and  spectral instability of the standing wave solutions  $e^{i\omega t}\mathbf{\Phi}(x)$ to the equation when the profile $\bb \Phi(x)$  has mixed structure (i.e. has  bumps and tails). In our approach  we essentially use the extension theory of symmetric operators by Krein-von Neumann, and the analytic perturbations theory. 
\end{abstract}

\begin{keyword}  logarithmic nonlinearity \sep Nonlinear Schr\"odinger equation\sep  orbital stability  \sep  spectral instability \sep standing wave \sep star graph.

\MSC[2010]
35Q55; 81Q35; 37K40; 37K45; 47E05\end{keyword}
\maketitle

\section{Introduction}

The logarithmic Schr\"odinger equation
$$
i\partial_tu(t,x)+\Delta u(t,x) + u(t,x)\log|u(t,x)|^2 = 0,\quad u(t,x): \mathbb{R}\times \mathbb{R}^n \rightarrow \mathbb{C},\,\, n\geq 1,$$
admits applications to quantum mechanics, quantum optics, nuclear physics, transport
and diffusion phenomena, open quantum systems, effective quantum gravity,
theory of superfluidity and Bose-Einstein condensation (BEC). This equation has been proposed by Bialynicki-Birula and Mycielski (see \cite{BiaMyc76})
in order to obtain a nonlinear equation which helped to quantify departures
from the strictly linear regime, preserving in any number of dimensions some
fundamental aspects of quantum mechanics, such as separability and additivity
of total energy of noninteracting subsystems. 

In the present paper we study the logarithmic Schr\"odinger equation on a star graph
 $\mathcal{G}$,  i.e. $N$
half-lines $(0,\infty)$ joined at the vertex $\nu=0$. Namely, on $\mathcal{G}$ we consider the   following nonlinear Schr\"odinger equation   with $\delta$-interaction (NLS-log-$\delta$ equation)
\begin{equation}\label{NLS_graph_ger}
i\partial_t \mathbf{U}(t,x)-\mathbf{H}^\alpha_\delta\mathbf{U}(t,x) +\mathbf{U}(t,x)\log|\mathbf{U}(t,x)|^2=0,
\end{equation}
where $\mathbf{U}(t,x)=(u_j(t,x))_{j=1}^N:\mathbb{R}\times \mathbb{R}_+\rightarrow \mathbb{C}^N$,\,  nonlinearity acts componentwise,  i.e. $(\mathbf{U}\log|\mathbf{U}|^2)_j=u_j\log|u_j|^2,$ and $\mathbf{H}^\alpha_\delta$ is the self-adjoint operator on $L^2(\mathcal{G})$ defined  for $\mathbf{V}=(v_j)_{j=1}^N$ by
\begin{equation}\label{D_alpha}
\begin{split}
(\mathbf{H}^\alpha_\delta \mathbf{V})(x)&=\left(-v_j''(x)\right)_{j=1}^N,\quad x> 0,\\
\dom(\mathbf{H}^\alpha_\delta)&=\left\{\mathbf{V}\in H^2(\mathcal{G}): v_1(0)=\dots=v_N(0),\,\,\sum\limits_{j=1}^N  v_j'(0)=\alpha v_1(0)\right\}.
\end{split}
\end{equation}
The condition \eqref{D_alpha} can be considered  as an analog of  $\delta$-interaction condition  for the \newline
Schr\"odinger operator on the line  (see \cite{AlbGes05}), which justifies the  name of the equation. 

Equation  \eqref{NLS_graph_ger} means that on each edge of the graph (i.e. on each half-line) we have 
$$i\partial_t u_j(t,x)+\partial_x^2u_j(t,x) +u_j(t,x)\log|u_j(t,x)|^2=0,\,\,x>0,\,\,j\in\{1,\dots,N\},$$ 
moreover, the vectors $\bb U(t,0)=(u_j(t,0))_{j=1}^N$ and $\bb U'(t,0)=(u'_j(t,0))_{j=1}^N$ satisfy conditions in \eqref{D_alpha}.

 Equation \eqref{NLS_graph_ger} models propagation through junctions in networks. 
In particular, models of BEC on graphs/networks is a topic of active research.

The quantum graphs (star graphs equipped with  a linear  Hamiltonian $\bb H$) have been a very developed subject in the last couple of decades. They give  simplified models in mathematics, physics, chemistry, and engineering, when one considers propagation of waves of various type through a quasi one-dimensional  system that looks like a thing neighborhood of a  graph (see \cite{BK, Exn08, K, Mug15} for details and references). 

The nonlinear PDEs on graphs have been actively  studied in the last ten years in the context of existence, stability, and propagation of solitary waves (see \cite{CacFin17, Noj14} for the references).

The main purpose  of this work is to study the   stability properties of the standing wave solutions   
$$
\mathbf{U}(t,x)=e^{i\omega t}\mathbf{\Phi}(x)=\left(e^{i\omega t}\varphi_{j}(x)\right)_{j=1}^N,
$$ 
to NLS-log-$\delta$ equation   \eqref{NLS_graph_ger}. Analogous problem was considered for NLS-log-$\delta$ equation on the line (see \cite{AngArd17, AngGol17}).
Similarly to the case of NLS equation with power nonlinearity investigated  in  \cite{AdaNoj14},  it can be  shown that all possible profiles $\bb \Phi(x)$ belong to the specific family of $\left[\frac{N-1}{2}\right]+1$  vector functions $\bb\Phi_k^\alpha=(\varphi^\alpha_{k,j})_{j=1}^N, \,\,k=0,\dots,\left[\tfrac{N-1}{2}\right],$ given by 
\begin{equation}\label{Phi_vect_log}
\begin{split}
 \varphi_{k,j}^\alpha(x)= \left\{
                    \begin{array}{ll}
                      e^{\tfrac{\omega+1}{2}}e^{-\tfrac{(x-a_k)^2}{2}}, & \,\hbox{$j=1,\dots,k$;} \\
                      e^{\tfrac{\omega+1}{2}}e^{-\tfrac{(x+a_k)^2}{2}}, & \,\hbox{$j=k+1,\dots,N,$}
                    \end{array}
                  \right.
                \,\,\text{where}\,\, a_k=\frac{\alpha}{2k-N}.
                  \end{split}
\end{equation} 

In the case $\alpha<0$ vector  $\mathbf{\Phi}_k^\alpha=(\varphi^\alpha_{k,j})_{j=1}^N$ has $k$ \textit{bumps} and $N-k$ \textit{tails} (all possible profiles for $N=5$ are given on Figure 1). It is easily seen that  
 $\mathbf{\Phi}^{\alpha}_0$ is the \textit{N-tail profile}. Moreover, the $N$-tail profile is the only symmetric (i.e. invariant under permutations of the edges) profile.  
  In the case $\alpha>0$ vector  $\mathbf{\Phi}_k^\alpha=(\varphi^\alpha_{k,j})_{j=1}^N$ has $k$ \textit{tails} and $N-k$ \textit{bumps} respectively. The case $N=5$ is  demonstrated on Figure 2.
  \\ 
  
 	\begin{tikzpicture}[scale=0.8]
	\clip (0,-3) rectangle (18,2);
	    \draw[-,color=gray] (2,1).. controls +(-0.7,-0.7) ..  (0.3,0.2);
       \draw[-,color=gray] (2,1).. controls +(-0.3,-0.9)  ..  (0.5,-1);  
        \draw[-,color=gray] (2,1).. controls +(0.2,-1)  ..  (2.8,-1.2);
        \draw[-,color=gray] (2,1) .. controls +(0.2,-0.5)  ..  (4.2,-0.2);
            \draw[-,color=gray] (2,1).. controls +(0.6,-0.3)  ..  (3.5,1);
		\draw[-latex,thin](2,0)--++(-2,0);
        \draw[-latex, thin](2,0)--++(-1.8,-1.3);
        \draw[-latex, thin](2,0)--++(1,-1.7);
		\draw[-latex, thin](2,0)--++(2.5,-0.5);
        \draw[-latex, thin](2,0)--++(2,1);

		\begin{scope}[shift={(6,0)}]
	     \draw[-,color=gray] (2,1).. controls +(-0.5,0.5) and +(1.5,-0.1) ..  (0.3,0.2);
       \draw[-,color=gray] (2,1).. controls +(-0.3,-0.9)  ..  (0.5,-1);  
        \draw[-,color=gray] (2,1).. controls +(0.2,-1)  ..  (2.8,-1.2);
        \draw[-,color=gray] (2,1) .. controls +(0.2,-0.5)  ..  (4.2,-0.2);
            \draw[-,color=gray] (2,1).. controls +(0.6,-0.3)  ..  (3.5,1);
		\draw[-latex,thin](2,0)--++(-2,0);
        \draw[-latex, thin](2,0)--++(-1.8,-1.3);
        \draw[-latex, thin](2,0)--++(1,-1.7);
		\draw[-latex, thin](2,0)--++(2.5,-0.5);
        \draw[-latex, thin](2,0)--++(2,1);
		\node[label={[xshift=2cm, yshift=-2.5cm] Figure 1}]{};	\end{scope}

		\begin{scope}[shift={(12,0)}]
	     \draw[-,color=gray] (2,1).. controls +(-0.5,0.5) and +(1.5,-0.1) ..  (0.3,0.2);
       \draw[-,color=gray] (2,1).. controls +(-0.3,-0.9)  ..  (0.5,-1);  
        \draw[-,color=gray] (2,1).. controls +(0.2,-1)  ..  (2.8,-1.2);
        \draw[-,color=gray] (2,1) .. controls +(0.2,-0.5)  ..  (4.2,-0.2);
            \draw[-,color=gray] (2,1).. controls +(0.3,0.5) and +(-0.8,-0.5)  ..  (3.5,1);
		\draw[-latex,thin](2,0)--++(-2,0);
        \draw[-latex, thin](2,0)--++(-1.8,-1.3);
        \draw[-latex, thin](2,0)--++(1,-1.7);
		\draw[-latex, thin](2,0)--++(2.5,-0.5);
        \draw[-latex, thin](2,0)--++(2,1);
	\end{scope}
\end{tikzpicture}

\begin{tikzpicture}[scale=0.8]
	\clip (0,-3) rectangle (18,2);
	    \draw[-,color=gray] (2,1).. controls +(-0.5,0.5) and +(1.5,-0.1) ..  (0.3,0.2);
       \draw[-,color=gray] (2,1).. controls +(-0.3,0.4) and +(1, 0.5) ..  (0.5,-1);  
        \draw[-,color=gray] (2,1).. controls +(0.2,0.3) and +(-0.1, -0.1) ..  (2.8,-1.2);
        \draw[-,color=gray] (2,1) .. controls +(0.2,0.5) and +(-1.1, 0.3) ..  (4.2,-0.2);
            \draw[-,color=gray] (2,1)..  controls +(0.3,0.5) and +(-0.8,-0.5)  ..  (3.5,1);
		\draw[-latex,thin](2,0)--++(-2,0);
        \draw[-latex, thin](2,0)--++(-1.8,-1.3);
        \draw[-latex, thin](2,0)--++(1,-1.7);
		\draw[-latex, thin](2,0)--++(2.5,-0.5);
        \draw[-latex, thin](2,0)--++(2,1);

		\begin{scope}[shift={(6,0)}]
	     \draw[-,color=gray] (2,1).. controls +(-0.7,-0.7) ..  (0.3,0.2);
       \draw[-,color=gray] (2,1).. controls +(-0.3,0.4) and +(1, 0.5) ..  (0.5,-1);  
        \draw[-,color=gray] (2,1).. controls +(0.2,0.3) and +(-0.1, -0.1) ..  (2.8,-1.2);
        \draw[-,color=gray] (2,1) .. controls +(0.2,0.5) and +(-1.1, 0.3) ..  (4.2,-0.2);
           \draw[-,color=gray] (2,1)..  controls +(0.3,0.5) and +(-0.8,-0.5)  ..  (3.5,1);
		\draw[-latex,thin](2,0)--++(-2,0);
        \draw[-latex, thin](2,0)--++(-1.8,-1.3);
        \draw[-latex, thin](2,0)--++(1,-1.7);
		\draw[-latex, thin](2,0)--++(2.5,-0.5);
        \draw[-latex, thin](2,0)--++(2,1);
		\node[label={[xshift=2cm, yshift=-2.5cm] Figure 2}]{};	\end{scope}

		\begin{scope}[shift={(12,0)}]
	    \draw[-,color=gray] (2,1).. controls +(-0.7,-0.7) ..  (0.3,0.2);
       \draw[-,color=gray] (2,1).. controls +(-0.3,0.4) and +(1, 0.5) ..  (0.5,-1);  
        \draw[-,color=gray] (2,1).. controls +(0.2,0.3) and +(-0.1, -0.1) ..  (2.8,-1.2);
        \draw[-,color=gray] (2,1) .. controls +(0.2,0.5) and +(-1.1, 0.3) ..  (4.2,-0.2);
          \draw[-,color=gray] (2,1).. controls +(0.6,-0.3)  ..  (3.5,1);
		\draw[-latex,thin](2,0)--++(-2,0);
        \draw[-latex, thin](2,0)--++(-1.8,-1.3);
        \draw[-latex, thin](2,0)--++(1,-1.7);
		\draw[-latex, thin](2,0)--++(2.5,-0.5);
        \draw[-latex, thin](2,0)--++(2,1);
	\end{scope}
\end{tikzpicture}

 In \cite{Ard16} the author proved (via variational approach) the orbital  stability of the symmetric profile $\bb\Phi_0^\alpha$ in the energy space $W_{\EE}(\mathcal{G})$ (defined in notation section  below) under the restriction  $\alpha<\alpha^*<0$. Namely, the orbital stability follows from the fact that $\bb\Phi_0^\alpha$ is a  minimizer of the action functional  restricted to the Nehari manifold. 
Recently in  \cite{AngGol16} we  proved  orbital stability of $\bb\Phi_0^\alpha$  for  $\alpha<0$ in the weighted Hilbert space $W^1_{\EE}(\mathcal{G})$ without restriction $\alpha<\alpha^*<0$. The use of different space $W^1_{\EE}(\mathcal{G})$ is due to application of the classical Lyapunov  linearization procedure around the standing wave solution (i.e. the linearized  operator associated to  $\bb\Phi_0^\alpha$  has to be rigorously defined in an appropriate Hilbert space). 

The main result of this paper is the following stability/instability theorem for  the rest of the profiles   $\mathbf{\Phi}^{\alpha}_k,\  k\neq 0$. 

\begin{theorem}\label{main_log}
For any $k=1,\dots,\left[\tfrac{N-1}{2}\right]$, $\omega\in\mathbb{R},$ the standing wave $e^{i\omega t}\mathbf{\Phi}^{\alpha}_k$ is spectrally unstable for $\alpha<0$ and orbitally stable in $W^1_{\EE,k}(\mathcal{G})$ for $\alpha>0$.
\end{theorem}

 To our knowledge, this is the first result on the stability/instability of the profiles $\mathbf{\Phi}_{k}^\alpha$ in the case $k\neq 0$. For $\alpha<0$ they are called \textit{excited states }due to the property $S(\mathbf{\Phi}^{\alpha}_0)<S(\mathbf{\Phi}^{\alpha}_k)<S(\mathbf{\Phi}^{\alpha}_{k+1})$, where $S$ is the action functional associated to  equation \eqref{NLS_graph_ger}. Stability  of the excited states is itself very interesting problem since  there
are only few cases when excited states of NLS equations are explicitly known. 

It is worth noting that we do not use variational techniques. Our approach  is purely analytical, and it is based on  the extension theory of symmetric operators, the  analytic perturbations theory,  and  the  well-known approach by  Grillakis,  Shatah and Strauss. 
\vskip0.2in

\noindent{\bf Notation.}
Let $L$ be a densely defined symmetric operator in a Hilbert space $\mathcal{H}$. The domain of $L$ is denoted by $\dom(L)$. The \textit{deficiency subspaces} and the \textit{deficiency  numbers} of $L$ are denoted by $\mathcal{N}_{\pm}(L):=\ker(L^*\mp iI)$ and  $n_\pm(L):=\dim \mathcal{N}_{\pm}(L)$ respectively. The number of negative eigenvalues of $L$ (counting multiplicities) is denoted by  $n(L)$ (\textit{the Morse index}). The spectrum  and the resolvent set of $L$ are  denoted by $\sigma(L)$ and $\rho(L)$ respectively.

  We denote by  $\mathcal{G}$ the star graph constituted by $N$ half-lines attached to a common vertex $\nu=0$. On the graph we define 
  \begin{equation*}
  L^2(\mathcal{G})=\bigoplus\limits_{j=1}^NL^2(\mathbb{R}_+),\quad H^1(\mathcal{G})=\bigoplus\limits_{j=1}^NH^1(\mathbb{R}_+),\quad H^2(\mathcal{G})=\bigoplus\limits_{j=1}^NH^2(\mathbb{R}_+). 
\end{equation*}  
For instance, the norm in $L^2(\mathcal{G})$ is defined by $$||\bb V||^2_{L^2(\mathcal{G})}=\sum\limits_{j=1}^N||v_j||^2_{L^2(\mathbb{R}_+)},\quad \mathbf{V}=(v_j)_{j=1}^N.$$  
 By  $||\cdot||_2$ and $(\cdot,\cdot)$ we will  denote  the norm  and   the scalar product in  $L^2(\mathcal{G})$.

We also denote by  $L_k^2(\mathcal{G})$ and $\EE$ the spaces 
\begin{equation*}
\begin{split}
L_k^2(\mathcal{G})=\left\{\begin{array}{c}\bb V=(v_j)_{j=1}^N\in L^2(\mathcal{G}): v_1(x)=\dots=v_k(x),\\
v_{k+1}(x)=\dots=v_N(x),\, x>0
\end{array}\right\}, 
\end{split}
\end{equation*}
$$\mathcal{E}=\{\bb V\in H^1(\mathcal{G}): \,  v_1(0)=\dots=v_N(0)\}.$$
On $\mathcal{G}$ we define the following  weighted Hilbert spaces  
 \begin{equation*}
 W^j(\mathcal{G})=\bigoplus\limits_{j=1}^NW^j(\mathbb{R}_+),\quad  W^j(\mathbb{R}_+)=\{f\in H^j(\mathbb{R}_+): x^jf\in L^2(\mathbb{R}_+)\}, \,\, j\in\{1,2\}, 
 \end{equation*}
and the Banach space  $$W(\mathcal{G})=\bigoplus_{j=1}^N W(\mathbb{R}_+),\,\,\, \text{where}\,\,\,W(\mathbb{R}_+)=\{f\in H^1(\mathbb{R_+}): |f|^2 \log|f|^2\in L^1(\mathbb R_+)\}.$$ 

\noindent Using the above notations, we define $$W^1_{\mathcal{E}}(\mathcal{G})=W^1(\mathcal{G})\cap \mathcal{E},\quad W^1_{\mathcal{E},k}(\mathcal{G})=W^1_{\mathcal{E}}(\mathcal{G})\cap L^2_k(\mathcal{G}), \quad W_{\mathcal{E}}(\mathcal{G})=W(\mathcal{G})\cap\mathcal{E}.$$

\section{Well-posedness}
In this section we prove the well-posedness of the Cauchy problem for \eqref{NLS_graph_ger}
in the space $W^1_\mathcal{E}(\mathcal{G})$. 
In \cite{Ard16} the well-posedness was proved in the Banach space $W_\mathcal{E}(\mathcal{G})$. Namely, the author showed the following result.
\begin{proposition}\label{well_log_graph}  For any $\bb U_0\in W_{\mathcal{E}}(\mathcal{G})$ there is a unique  solution \newline $\bb U \in C(\mathbb{R}, W_{\mathcal{E}}(\mathcal{G}))\cap C^1(\mathbb{R}, W'_{\mathcal E}(\mathcal{G}))$ of \eqref{NLS_graph_ger} such that $\bb U(0) = \bb U_0$ and $\sup\limits_{t\in\mathbb{R}}||\bb U(t)||_{W_{\mathcal{E}}(\mathcal{G})} <\infty$. Furthermore, the conservation of energy and charge holds, that is,
$$E(\bb U(t)) = E(\bb U_0),\,\,\, \text{and}\,\,\, Q(\bb U(t))=||\bb U(t)||^2_2=||\bb U_0||^2_2,$$
where the energy $E$ is defined by 
 \begin{equation}\label{energy}
E(\bb V)=\tfrac 1{2}||\bb V'||^2_2-\tfrac 1{2}\sum\limits_{j=1}^N\int\limits_0^\infty|v_j|^2\logg|v_j|^2dx +\tfrac{\alpha}{2}|v_1(0)|^2, \quad \mathbf{V}=(v_j)_{j=1}^N\in W_{\mathcal E}(\mathcal{G}).
\end{equation}
\end{proposition}

Using the above result, we obtain the well-posedness in $W^1_{\mathcal{E}}(\mathcal{G})$. 
\begin{theorem}\label{well_pos_log}
If $\bb U_0\in W^1_{\mathcal{E}}(\mathcal{G})$, there is a unique solution $\bb U(t)$ of \eqref{NLS_graph_ger} such that $\bb U(t)\in C(\mathbb R, W^1_{\mathcal{E}}(\mathcal{G}))$ and $\bb U(0)=\bb U_0$.  Furthermore, the conservation of energy and charge holds. Moreover, if  $\bb U_0\in W^1_{\mathcal{E},k}(\mathcal{G})$, then the solution $\bb U(t)$ to the Cauchy problem   also belongs to  $ W^1_{\mathcal{E},k}(\mathcal{G})$.
\end{theorem}
\begin{proof}
By \cite[Lemma 3.1]{AngGol17}, we get $W^1_{\mathcal{E}}(\mathcal{G})\subset W_{\mathcal{E}}(\mathcal{G})$, and, therefore, $\bb U_0\in W_{\mathcal{E}}(\mathcal{G})$. By Proposition \ref{well_log_graph}, we get the uniqueness of the solution in $W_{\mathcal{E}}(\mathcal{G})$ and the conservation of energy and charge. Note also that the solution $\bb U(t)$ belongs to $W^1_{\mathcal{E}}(\mathcal{G})$ due to \cite[Lemma 7.6.2]{Caz98}.

Let us prove the continuity in $t$.   Assume that $t_n\underset{n\rightarrow\infty}{\longrightarrow} t$. Arguing as in the proof of \cite[Theorem 2.1]{CazHar80}, we can show that  $\bb U(t)\in C(\mathbb{R}, H^1(\mathcal{G}))$. Hence, $\bb U(t_n)\underset{n\rightarrow\infty}{\longrightarrow} \bb U(t)$ in $H^1(\mathcal{G})$ and we can  assume that $\bb U(t_n)\underset{n\rightarrow\infty}{\longrightarrow} \bb U(t)$ a.e. in $\mathcal{G}$.
It remains to prove $\bb U(t)\in C(\mathbb R, L^2(x^2, \mathcal{G}))$ ($L^2(x^2, \mathcal{G})$ denotes weighted $L^2$-space).  By \cite[Lemma 7.6.2]{Caz98}, we get that the function $t\mapsto ||x\bb U(t)||_{2}^2$ is continuous on $\mathbb{R}$. Thus, $||x\bb U(t_n)||_{2}^2\underset{n\rightarrow\infty}{\longrightarrow}||x\bb U(t)||_{2}^2$. Having additionally almost  everywhere convergence of $x\bb U(t_n)$ to $x\bb U(t)$, we get from Brezis-Lieb Lemma in \cite{BL}
$$
||x\bb U(t_n)-x\bb U(t)||_{2}^2\underset{n\rightarrow\infty}{\longrightarrow}0,
$$
which implies $\bb U(t)\in  C(\mathbb R, W^1_{\mathcal{E}}(\mathcal{G})).$

The last assertion of the theorem follows  as the solution to the Cauchy problem for \eqref{NLS_graph_ger} was obtained by approximation procedure in \cite{Ard16} (approximating sequence consists of the solutions to the Cauchy problem for the reduced nonlinear equation with Lipschitz  continuous nonlinearity) and the  evolution group $e^{-it\bb H_\delta^\alpha}$ preserves the space $$\mathcal{E}_k=\mathcal{E}\cap L^2_k(\mathcal{G}).$$ See the proof of Lemma 2.3  in \cite{AngGol17a} (or Theorem 3.4 in \cite{AngGol16}) for the detailed  explanation of this fact. 

\end{proof}

The properties of the energy functional $E$ are essential for the investigation of the orbital stability. For example, in $\mathcal{E}$ the energy  functional is ill-defined. From \cite[Lemma 2.6]{Caz83} it follows that $E$ is continuously differentiable in $W_{\mathcal{E}}(\mathcal{G})$. Below we prove its continuity in the space $W^1_{\mathcal{E}}(\mathcal{G})$.
\begin{proposition}\label{cont_ener}
The energy functional $E$  defined by \eqref{energy} is continuous in $W^1_{\mathcal{E}}(\mathcal{G})$. 
\end{proposition}
\begin{proof}
Let $\bb V_n\underset{n\rightarrow\infty}{\longrightarrow} \bb V$ in $W^1_{\mathcal{E}}(\mathcal{G})$. It is easily seen that $$||\bb V'_n||_2^2+\alpha|v_{1,n}(0)|^2\underset{n\rightarrow\infty}{\longrightarrow}||\bb V'||_2^2+\alpha|v_{1}(0)|^2.$$
Hence, we need to prove the  continuity of the nonlinearity part of the functional $E$.
Basically we will use the following inequality (see \cite[Lemma 2.4.3]{CazHar80}) for  $|f|\geq |g|$
\begin{equation}\label{base_log}
\left||f|^2\log|f|^2-|g|^2\log|g|^2\right|\leq \left(1+|\log|f|^2|\right)\left||f|^2-|g|^2\right|. 
\end{equation}
To simplify the notation for the sets we  write $\{\text{condition}\}$ instead of $\{x\in\mathbb{R}:\text{condition}\}$.
From \eqref{base_log} we obtain
\begin{equation}\label{base_log1}
\begin{split}
&\int\limits_{\mathbb{R}_+}\left||v_{j,n}|^2\log|v_{j,n}|^2-|v_{j}|^2\log|v_{j}|^2\right|dx\leq \int\limits_{\{|v_{j,n}|\geq |v_{j}|\}}\left(1+|\log|v_{j,n}|^2|\right)\left||v_{j,n}|^2-|v_{j}|^2\right|dx \\&+ \int\limits_{\{|v_{j,n}|\leq |v_{j}|\}}\left(1+|\log|v_{j}|^2|\right)\left||v_{j,n}|^2-|v_{j}|^2\right|dx.
\end{split}
\end{equation}
Let us show that the right hand side of \eqref{base_log1} tends to zero. We will estimate the first one expression of the right hand side of \eqref{base_log1} since the analysis for the second one is analogous. 
By the Cauchy-Schwarz inequality we get
\begin{equation}
\begin{split}\label{base_log2}
&\int\limits_{\{|v_{j,n}|\geq |v_{j}|\}}\left(1+|\log|v_{j,n}|^2|\right)\left||v_{j,n}|^2-|v_{j}|^2\right|dx\\&\leq\left(\int\limits_{\{|v_{j,n}|\geq |v_{j}|\}}\left(1+|\log|v_{j,n}|^2|\right)^2(|v_{j,n}|+|v_{j}|)^2dx \right)^{\tfrac1{2}} \left(\int\limits_{\mathbb{R}_+}(|v_{j,n}|-|v_{j}|)^2dx\right)^{\tfrac1{2}}.
\end{split}
\end{equation}
 Let $\varepsilon>0$ be  such that $|\log|v_{j,n}|^2|\leq \frac{1}{|v_{j,n}|^{1/2}}$ for $|v_{j,n}|<\varepsilon$. Observing that the set $\{|v_{j,n}|\geq \varepsilon\}$ is contained in some bounded interval (due to $v_{j,n}\in H^1(\mathbb{R}_+)$) and recalling that $v_{j,n}\in L^1(\mathbb{R}_+)$ (see the proof of Lemma 3.1 in \cite{AngGol17}), we obtain
 \begin{equation*}
 \begin{split}
 &\int\limits_{\{|v_{j,n}|\geq |v_{j}|\}}\left(1+|\log|v_{j,n}|^2|\right)^2(|v_{j,n}|+|v_{j}|)^2dx\leq \int\limits_{\{|v_{j,n}|\geq |v_{j}|\}\cap\{|v_{j,n}|\geq \varepsilon\}}C_1|v_{j,n}|^2dx\\&+\int\limits_{\{|v_{j,n}|\geq |v_{j}|\}\cap\{|v_{j,n}|<\varepsilon\}}\left(1+\frac{1}{|v_{j,n}|^{1/2}}\right)^24|v_{j,n}|^2dx\\&\leq C_2+4\int\limits_{\mathbb{R}_+}\left(|v_{j,n}|^2+2|v_{j,n}|^{3/2}+|v_{j,n}|\right)dx<\infty,
 \end{split}
 \end{equation*}
 where $C_1$ and $C_2$  are positive constants. It is worth noting that $\int\limits_{\mathbb{R}_+}|v_{j,n}|^{3/2}dx<\infty$ follows from the Cauchy-Schwarz inequality and the inclusion $v_{j,n}\in L^1(\mathbb{R}_+)$.
 
 Finally, from  \eqref{base_log1}-\eqref{base_log2} we get 
 $$\int\limits_{\mathbb{R}_+}|v_{j,n}|^2\log|v_{j,n}|^2dx \underset{n\rightarrow\infty}{\longrightarrow} \int\limits_{\mathbb{R}_+}|v_{j}|^2\log|v_{j}|^2dx.$$ 
\end{proof}

\section{The proof of the main theorem}

 Crucial role in
the orbital stability analysis is played by the symmetries of  NLS equation \eqref{NLS_graph_ger}. The basic symmetry associated to the  mentioned equation is  phase-invariance (in particular, translational invariance  does  not hold due to the defect). Thus,  it is reasonable  to define orbital stability as follows.  
\begin{definition}\label{def_stabil}
 The standing wave $\mathbf U(t, x) = e^{i\omega t}\mathbf{\Phi}(x)$ is said to be \textit{orbitally stable} in a Banach space $X$  if for any $\varepsilon > 0$ there exists $\eta > 0$ with the following property: if $\mathbf U_0 \in X$  satisfies $||\mathbf U_0-\mathbf{\Phi}||_{X} <\eta$,
then the solution $\mathbf U(t)$ of \eqref{NLS_graph_ger} with $\mathbf U(0) = \mathbf U_0$ exists for any $t\in \mathbb{R}$, and
\[\sup\limits_{t\in \mathbb{R}}\inf\limits_{\theta\in\mathbb{R}}||\mathbf U(t)-e^{i\theta}\mathbf{\Phi}||_{X} < \varepsilon.\]
Otherwise, the standing wave $\mathbf U(t, x) = e^{i\omega t}\mathbf{\Phi}(x)$ is said to be \textit{orbitally unstable} in $X$.
\end{definition}
Below we will define spectral stability/instability of  $e^{i\omega t}\mathbf{\Phi}(x)$. We assume that $\bb \Phi(x)$ belongs to the family of profiles defined by  \eqref{Phi_vect_log}.  Change of variables $\bb U(t,x) = e^{i\omega t}(\bb \Phi(x) + \bb V(t,x))$ in \eqref{NLS_graph_ger}  leads to 
\begin{equation}\label{linear_A}
\partial_t \bb V(t,x) = \bb A\bb V(t,x) + F(\bb V(t,x)),\end{equation}
where $\bb A$ is the linearized operator defined by 
\begin{equation*}\label{def:A}
\bb A\bb V=-i\left\{\bb H_\delta^\alpha\bb V+\omega \bb V
-\bb V
\log\bb\Phi^2-\bb V-\overline{\bb V}\right\},
\end{equation*}
and  $F(\bb V)$ is the nonlinear term given by 
$$F(\bb V)=i\left\{(\bb\Phi+\bb V)\log|\bb\Phi+\bb V|^2-   \bb V
\log\bb\Phi^2 -\bb V-\overline{\bb V}- \bb \Phi\log\bb \Phi^2 \right\}.$$

\begin{definition}
The standing wave $\mathbf U(t, x) = e^{i\omega t}\mathbf{\Phi}(x)$ is said to be \textit{spectrally stable} if $$\sigma(\bb A)\subset i\mathbb{R}.$$
Otherwise, the standing wave $\mathbf U(t, x) = e^{i\omega t}\mathbf{\Phi}(x)$ is said to be \textit{spectrally unstable}.
\end{definition}

It is standard to show that $\sigma(\bb A)$ is symmetric with respect  to the real and imaginary axes (see, for instance, \cite[Lemma 5.6]{GrilSha90}). Hence,  it is equivalent to say that  $e^{i\omega t}\mathbf{\Phi}(x)$ is spectrally unstable if $\sigma(\bb A)$ contains some point $\lambda$ with  $\Re(\lambda)>0.$

 It is widely known  that the spectral instability is  a key prerequisite to show nonlinear (orbital) instability in numerous works (see \cite{GrilSha90, ShaStr00} and references therein). However, it is highly nontrivial problem whether spectral instability implies orbital instability. 
\subsection{Stability framework}
 
 In this subsection we introduce  basic objects of the stability framework. 
The action functional for NLS-log-$\delta$ equation is given by
 \begin{equation*}\label{act_log}  S(\bb V)=\tfrac 1{2}||\bb V'||^2_2+\tfrac{(\omega+1)}{2}||\bb V||^2_2 -\tfrac 1{2}\sum\limits_{j=1}^N\int\limits_0^\infty|v_j|^2\log|v_j|^2dx+\tfrac{\alpha}{2}|v_1(0)|^2,\quad \bb V\in W_{\mathcal{E}}(\mathcal{G}).
\end{equation*}
The profiles $\bb\Phi_k^\alpha$  defined by \eqref{Phi_vect_log} are the critical points of the action functional.
 Indeed, for $\bb U=(u_j)_{j=1}^N, \bb V=(v_j)_{j=1}^N\in W_{\mathcal{E}}(\mathcal{G})$,
 \begin{equation*}
 \begin{split}
& S'(\bb U)\bb V=\frac{d}{dt}S(\bb U+t\bb V)|_{t=0}\\&=\re\left[\sum\limits_{j=1}^N\left(\,\int\limits_{\mathbb{R}_+}u'_j\overline{v'_j}dx-\int\limits_{\mathbb{R}_+}u_j\overline{v_j}(\log|u_j|^2-\omega)dx\right)+\alpha u_1(0)\overline{v_1(0)}\right].
 \end{split}
 \end{equation*}
 Obviously $S'(\bb\Phi^\alpha_k)=\bb 0$.
 Below we will use the notation $\bb \Phi_k:=\bb\Phi_k^\alpha$.
 
 The basic ingredient of stability study is the operator $\bb A_k$ introduced in \eqref{linear_A} (index $k$ means that we need to linearize equation around each $\bb \Phi_k$). 
  To express $\bb A_k$ it is convenient to split $\bb V\in W^1_{\mathcal{E}}(\mathcal{G})$ into real and imaginary parts: $\bb V=\bb V_1+i\bb V_2$. Then  we get 
 $$\bb A_k\bb V= \bb T_{2,k}^\alpha\bb V_2-i\bb T_{1,k}^\alpha\bb V_1,$$
 where
\begin{equation*}\label{T_k}
\begin{split}
&\bb T_{1,k}^\alpha=\diag\left(\underset{\bf 1}{-\frac{d^2}{dx^2}+V_{k}^-(x)},\dots,\underset{\bf k}{-\frac{d^2}{dx^2}+V_{k}^-(x)}, \underset{\bf k+1}{-\frac{d^2}{dx^2}+V_{k}^+(x)},\dots, \underset{\bf N}{-\frac{d^2}{dx^2}+V_{k}^+(x)}\right),\\
&\bb T_{2,k}^\alpha=\diag\left(\underset{\bf 1}{-\frac{d^2}{dx^2}+W_{k}^-(x)},\dots,\underset{\bf k}{-\frac{d^2}{dx^2}+W_{k}^-(x)}, \underset{\bf k+1}{-\frac{d^2}{dx^2}+W_{k}^+(x)},\dots, \underset{\bf N}{-\frac{d^2}{dx^2}+W_{k}^+(x)}\right),\\
&\dom(\bb T_{1,k}^\alpha)=\dom(\bb T_{2,k}^\alpha)=\left\{\mathbf{V}\in W^2(\mathcal{G}):  v_1(0)=\dots=v_N(0), \sum\limits_{j=1}^N  v_j'(0)=\alpha v_1(0) \right\},
\end{split}
\end{equation*}
with $V_{k}^{\pm}(x)= \left(x\pm\frac{\alpha}{2k-N}\right)^2-3$  and $W_{k}^{\pm}(x)= \left(x\pm\frac{\alpha}{2k-N}\right)^2-1$. 
Finally, we get formally \begin{equation*}
\bb A_k =\left(\begin{array}{cc}
\bb 0 & \bb I\\
-\bb I & \bb 0
\end{array}\right) \bb H_k^\alpha,
 \end{equation*} where 
$\bb H_{k}^\alpha=\left(\begin{array}{cc}
\bb T_{1,k}^\alpha & \bb 0\\
\bb 0& \bb T_{2,k}^\alpha
\end{array}\right)$, moreover, $\bb 0$ and  $\bb I$ are zero and identity $N\times N$ matrices. 
  Observe also that   $\bb H_k^\alpha$ is the self-adjoint operator associated with $S''(\bb \Phi_k)$ (see \cite[Section 6]{AdaNoj14} for details). 
  
 Noting that $\partial_\omega||\bb \Phi_k||_2^2>0$, and combining \cite[Theorem 3.5]{GrilSha87}  with \cite[Theorem 5.1]{GrilSha90}, we can formulate the stability/instability theorem  for the NLS-log-$\delta$ equation.
\begin{theorem}\label{stabil_graph_log}
 Let $\alpha\neq 0$,\, $k=1,\dots,\left[\tfrac{N-1}{2}\right]$,  and $n(\bb H^\alpha_{k}|_{L_k^2(\mathcal{G})})$ be the number of negative eigenvalues of $\bb H^\alpha_{k}$ in $L_k^2(\mathcal{G})$. Suppose also  that 
 
 $1)$\,\,$\ker(\bb T^\alpha_{2,k})=\Span\{\bb\Phi_k\}$,
 
 $2)$\,\,$\ker(\bb T^\alpha_{1,k})=\{\bb 0\}$, 
 
  $3)$\,\, the negative spectrum of $\bb T^\alpha_{1,k}$ and  $\bb T^\alpha_{2,k}$ consists of a finite number of negative eigenvalues (counting multiplicities),   
  
  $4)$\,\, the rest of  the  spectrum of $\bb T^\alpha_{1,k}$ and  $\bb T^\alpha_{2,k}$ is positive and bounded away from zero. 
  \
  Then the following assertions hold.
\begin{itemize}
  \item[$(i)$] If $n(\bb H^\alpha_{k}|_{L_k^2(\mathcal{G})})=1$, then   the standing wave $e^{i\omega t}\bb\Phi_k$  is orbitally  stable in $W^1_{\EE,k}(\mathcal{G})$.
   \item[$(ii)$]  If $n(\bb H^\alpha_{k}|_{L_k^2(\mathcal{G})})=2$, then   the standing wave $e^{i\omega t}\bb\Phi_k$  is spectrally  unstable.
\end{itemize}
\end{theorem}
\begin{remark}
\begin{itemize}
\item[$(i)$] Note that for the proof of the item $(i)$ continuity of the energy functional $E$ proved in Proposition \ref{cont_ener} is essential (see the proof of Theorem 3.5 in \cite{GrilSha87}).
\item[$(ii)$] In item $(ii)$ we affirm only spectral instability since to show orbital instability we need to prove some additional nontrivial properties   of NLS-log-$\delta$ equation, for instance, estimate (6.2) in \cite{GrilSha90} for the semigroup $e^{t\bb A_k}$ generated by $\bb A_k$, or the property that
the mapping data-solution associated to NLS-log-$\delta$ equation is of class $C^2$ around $\bb \Phi_k$ (see \cite[Section 2]{HenPer82} for the general idea and 
 \cite{AngNat16} for the particular application). However,  we conjecture that for the operator $\bb A_k$ so-called spectral mapping theorem holds (that is, $\sigma(e^{\bb A_k})=e^{\sigma(\bb A_k)}$) which would imply estimate (6.2) in \cite{GrilSha90} (see, for instance, the discussion in \cite{GeoOht10}). 
\end{itemize}
\end{remark}
\subsection{Spectral properties of $\bb T^\alpha_{1,k}$ and $\bb T^\alpha_{2,k}$}
Below we  describe  the spectrum of the operators $\mathbf{T}_{1,k}^\alpha$ and $\mathbf{T}_{2,k}^\alpha$  which will help us to verify the conditions of Theorem \ref{stabil_graph_log}. Our ideas are based on the extension theory of symmetric operators and the perturbation theory. 

 The main result of this subsection is the following. 

\begin{theorem}\label{7main}  Let $\alpha \neq 0$, $k=1,\dots,\left[\tfrac{N-1}{2}\right]$.  Then the following assertions hold. 
\begin{enumerate}
\item[(i)]  If $\alpha<0$,  then  $n(\bb H_k^\alpha|_{L^2_k(\mathcal G)})=2$.

\item[(ii)]  If  $\alpha>0$, then  $n(\bb H_k^\alpha|_{L^2_k(\mathcal G)})=1$.
\end{enumerate}
\end{theorem}

The proof of Theorem \ref{7main} is an immediate consequence of  Propositions \ref{grafoN2} and \ref{n(L_1)} below.

 \begin{proposition}\label{grafoN2}
Let $\alpha\neq 0$,  $k=1,\dots,\left[\tfrac{N-1}{2}\right]$. Then the following assertions hold. 
\begin{itemize}
\item[$(i)$]   $\ker(\mathbf{T}^\alpha_{2,k})=\Span\{\mathbf{\Phi}_{k}\}$ and $\mathbf{T}^\alpha_{2,k}\geq 0$.
 \item[$(ii)$] $\ker(\mathbf{T}^\alpha_{1,k})=\{\mathbf{0}\}$. 
    \item[$(iii)$] The spectrum of the operators $\mathbf{T}^\alpha_{1,k}$ and $\mathbf{T}^\alpha_{2,k}$  is discrete.
  \end{itemize}
\end{proposition}
\begin{proof}
$(i)$\, It is clear that  $\mathbf{\Phi}_k\in\ker(\mathbf{T}^\alpha_{2,k})$. To show the equality  $\ker(\mathbf{T}^\alpha_{2,k})=\Span\{\mathbf{\Phi}_k\}$  let us note that  any $\mathbf{V}=(v_j)_{j=1}^N\in W^2(\mathcal{G})$ satisfies the following identity 
\begin{equation*}
-v_j''+((x\mp\tfrac{\alpha}{2k-N})^2-1)v_j=
\frac{-1}{\varphi_{k,j}}\frac{d}{dx}\left[\varphi_{k,j}^2\frac{d}{dx}\left(\frac{v_j}{\varphi_{k,j}}\right)\right],\quad x>0,
\end{equation*}
where $\varphi_{k,j}=\varphi_{k,j}^\alpha$ is defined by \eqref{Phi_vect_log}, and the sign $-$ ($+$) corresponds to $j=1,\dots, k$ ($j=k+1,\dots,N$).
Thus, for $\mathbf{V}\in \dom(\bb T^\alpha_{2,k})$ we obtain 
\begin{equation*}\label{nonneg_graph'}
\begin{split}
&(\mathbf{T}^\alpha_{2,k}\mathbf{V},\mathbf{V})=
\sum\limits_{j=1}^N\int\limits^{\infty}_{0}(\varphi_{k,j})^2\left[\frac{d}{dx}\left(\frac{v_j}{\varphi_{k,j}}\right)\right]^2dx+
\sum\limits_{j=1}^N\left[-v_j'{v}_j+v_j^2\frac{\varphi_{k,j}'}{\varphi_{k,j}}\right]^{\infty}_{0}\\&=\sum\limits_{j=1}^N\int\limits^{\infty}_{0}(\varphi_{k,j})^2\left[\frac{d}{dx}\left(\frac{v_j}{\varphi_{k,j}}\right)\right]^2dx+\sum\limits_{j=1}^N\left[v_j'(0){v_j(0)}-v_j^2(0)\frac{\varphi_{k,j}'(0)}{\varphi_{k,j}(0)}\right].
\end{split}
\end{equation*}
Using boundary conditions \eqref{D_alpha}, we get
\begin{equation*}
\begin{split}
&\sum\limits_{j=1}^N\left[v_j'(0){v_j(0)}-v_j^2(0)\frac{\varphi_{k,j}'(0)}{\varphi_{k,j}(0)}\right]=0,
\end{split}
\end{equation*}
which induces $(\mathbf{T}^\alpha_{2,k}\bb V, \bb V)> 0$ for $\bb V\in \dom(\mathbf{T}^\alpha_{2,k})\setminus\Span\{\mathbf{\Phi}_k\}$. Therefore, $\ker(\mathbf{T}^\alpha_{2,k})=\Span\{\mathbf{\Phi}_k\}$. 
   
$(ii)$\,Concerning the kernel of $\mathbf{T}^\alpha_{1,k}$,  the only $L^2(\mathbb{R}_+)$-solution of the equation 
$$
 -v_j''+((x\mp\tfrac{\alpha}{2k-N})^2-3)v_j=0
 $$ 
 is $v_j=\varphi_{k,j}'$ up to a factor. Thus, any element of $\ker(\mathbf{T}^\alpha_{1,k})$ has the form $\mathbf{V}=(v_j)_{j=1}^N=(c_j\varphi_{k,j}')_{j=1}^N,\, c_j\in\mathbb{R}$. Continuity condition $v_1(0)=\dots= v_N(0)$ induces that $c_1=\dots=c_N$, i.e.
 \begin{equation*}\label{kernel}
  v_j(x)=c\left\{\begin{array}{ll}
                     -\varphi_{k,j}', & \quad\hbox{$j=1,\dots ,k$;} \\
                     \varphi_{k,j}', & \quad\hbox{$j=k+1,\dots ,N$}
                    \end{array}
                  \right., \quad c\in\mathbb{R}.
                  \end{equation*} 
                  Condition $\sum\limits_{j=1}^Nv_j'(0)=\alpha v_j(0)$ is equivalent to the equality $c\left(\tfrac{\alpha^2}{(N-2k)^2}-1\right)=c \tfrac{\alpha^2}{(N-2k)^2}$. The last one induces that $c=0$, and, therefore,  $\mathbf{V}\equiv\bb 0$.

$(iii)$\, With slight modifications we can repeat the proof of \cite[Theorem 3.1, Chapter II]{BerShu91} to show that the spectrum of $\mathbf{T}^\alpha_{1,k}$ is discrete since  $\lim\limits_{x\to +\infty}(x\mp\tfrac{\alpha}{2k-N})^2-3=+\infty,$
i.e. $\sigma(\mathbf{T}^\alpha_{1,k})=\sigma_p(\mathbf{T}^\alpha_{1,k})=\{\mu_{1,j}\}_{j\in\mathbb{N}}$. In particular, we have the following distribution of the eigenvalues
$$
\mu_{1,1}<\mu_{1,2}<\cdot\cdot\cdot<\mu_{1,j}<\cdot\cdot\cdot,
$$
with $\mu_{1,j}\to +\infty$ as $j\to +\infty$.
\end{proof}

 Below using the perturbation theory we will study  
 $n(\mathbf{T}_{1,k}^\alpha)$ in the space $L^2_k(\mathcal{G})$ for any $k=1,\dots,\left[\tfrac{N-1}{2}\right]$. The following lemma states the analyticity of the family  of operators $\mathbf{T}^\alpha_{1,k}$. 

 \begin{lemma}\label{analici} As a function of $\alpha$, $(\mathbf{T}^\alpha_{1,k})$ is  real-analytic family of self-adjoint operators of type (B) in the sense of Kato.
\end{lemma}
\begin{proof} By  \cite[Theorem VII-4.2]{kato}, it suffices to note that the family  of bilinear forms $(B^{\alpha}_{1,k})$ defined by
\begin{equation*}
\begin{split}
&B^{\alpha}_{1,k}(\bb F,\bb G)=\sum\limits_{j=1}^N\int\limits_{\mathbb{R}_+}f'_{j}g'_{j}dx+\sum\limits_{j=1}^k\int\limits_{\mathbb{R}_+}f_{j} g_{j} \left(\left(x-\tfrac{\alpha}{2k-N}\right)^2-3\right)dx\\&+\sum\limits_{j=k+1}^N\int\limits_{\mathbb{R}_+}f_{j} g_{j} \left(\left(x+\tfrac{\alpha}{2k-N}\right)^2-3\right)dx+
\alpha f_{1}(0)g_{1}(0)
\end{split}
\end{equation*}
  is   real-analytic  of type (B). 
  \end{proof}
Observe that  $\mathbf{T}^\alpha_{1,k}$ tends (in the generalized sense) to the following self-adjoint matrix Schr\"odinger operator on $L^2(\mathcal{G})$ with the  Kirchhoff condition at $\nu=0$  as $\alpha\to 0$  
 \begin{equation}\label{L^0_1}
 \begin{split}
 &\bb T^0_1=\left(\Big(-\frac{d^2}{dx^2}+x^2-3\Big)\delta_{i,j}\right), \\
 &\dom(\bb T^0_1)=\left\{\mathbf{V}\in W^2(\mathcal{G}): v_1(0)=\dots=v_N(0),\,\,\sum\limits_{j=1}^N  v_j'(0)=0\right\}.
 \end{split}
 \end{equation}
 As we intend to study negative spectrum of $\mathbf{T}^\alpha_{1,k}$, we first need to describe spectral properties of  $\bb T^0_1$. 
  \begin{theorem}\label{spect_L^0_1}
 Let $\bb T^0_1$ be defined by \eqref{L^0_1} and $k=1,\dots,\left[\tfrac{N-1}{2}\right]$. Then 
\begin{itemize}
  \item[$(i)$] $\ker(\bb T^0_1)=\Span\{\hat{\bb\Phi}_{0,1},\dots,\hat{\bb\Phi}_{0,N-1}\}$, where
   \begin{equation*}\label{Psi_j}
\hat{\mathbf{\Phi}}_{0,j}=(0,\dots,0,\underset{\bf j}{\varphi'_{0}},\underset{\bf j+1}{-\varphi'_{0}},0,\dots,0), \,\,\varphi_0=e^{-\tfrac{x^2}{2}}. 
 \end{equation*}
  \item[$(ii)$] In the space  $L^2_k(\mathcal{G})$ we have $\ker(\bb L^0_1)=\Span\{\mathbf{\widetilde{\Phi}}_{0,k}\}$, i.e. $\ker(\bb T^0_1|_{L^2_k(\mathcal{G})})=\Span\{\mathbf{\widetilde{\Phi}}_{0,k}\}$, where
  \begin{equation}\label{Psi_0}
\mathbf{\widetilde{\Phi}}_{0,k}=\left(\underset{\bf 1}{\tfrac{N-k}{k}\varphi'_{0}},\dots, \underset{\bf k}{\tfrac{N-k}{k}\varphi'_{0}},\underset{\bf k+1}{-\varphi'_{0}},\dots,\underset{\bf N}{-\varphi'_{0}}\right). 
 \end{equation}
  \item[$(iii)$]  The operator $\bb T^0_1$ has one simple negative eigenvalue in $L^2(\mathcal{G})$, i.e. $n(\bb T^0_1)=1$. Moreover, $\bb T^0_1$ has one simple negative eigenvalue in $L^2_k(\mathcal{G})$ for any $k$, i.e.  $n(\bb T^0_1|_{L^2_k(\mathcal{G})})=1$.
   \item[$(iv)$] The positive part of the spectrum of $\bb T^0_1$  is bounded away from zero. 
  \end{itemize}
\end{theorem}
\begin{proof}
The proof  can be found in  \cite{AngGol16}. We repeat it  for  the reader's convenience. 

$(i)$ The only $L^2(\mathbb{R}_+)$-solution to the equation 
$$
 -v''_j+(x^2-3)v_j=0
 $$ 
 is $v_j=\varphi_0'$ (up to a factor). Thus, any element of $\ker(\mathbf{T}^0_1)$ has the form $\mathbf{V}=(v_j)_{j=1}^N=(c_j\varphi'_0)_{j=1}^N,\, c_j\in\mathbb{R}$. It is easily seen that the  continuity condition is satisfied since $\varphi'_0(0)=0$. Condition $\sum\limits_{j=1}^Nv_j'(0)=0$  gives rise to $(N-1)$-dimensional kernel of $\mathbf{T}^0_1$. Finally, note  that functions $\hat{\mathbf{\Phi}}_{0,j},\, j=1,\dots, N-1,$ form basis there.  
 
  $(ii)$ Arguing as in the previous item, we can show that $\ker(\bb T^0_1)$ is one-dimensional in $L^2_k(\mathcal G)$, and it is spanned by $\mathbf{\widetilde{\Phi}}_{0,k}$.

  $(iii)$ Denote $\opt_0=\left(\Big(-\frac{d^2}{dx^2}+x^2-3\Big)\delta_{k,j}\right)$. First, we will show that the operator $\bb T_0$ defined by 
\begin{equation*}
\bb T_0=\opt_0,\,\,\dom(\bb T_0)=\left\{\mathbf{V}\in W^2(\mathcal{G}):  v_1(0)=\dots=v_N(0)=0, \sum\limits_{j=1}^N  v_j'(0)=0 \right\}
\end{equation*}
is non-negative. The proof follows from the  identity   
\begin{equation*}
-v_j''+(x^2-3)v_j=
\frac{-1}{\varphi'_{0}}\frac{d}{dx}\left[(\varphi'_{0})^2\frac{d}{dx}\left(\frac{v_j}{\varphi'_{0}}\right)\right],\quad x> 0,
\end{equation*}
for any  $\mathbf{V}=(v_j)_{j=1}^N\in W^2(\mathcal{G})$.
Using the above equality and integrating by parts, we get for $\mathbf{V}\in \dom(\mathbf{T}_0)$
 \begin{equation*}\label{nonneg_graph}\begin{split}
&(\mathbf{T}_0\mathbf{V},\mathbf{V})=
\sum\limits_{j=1}^N\int\limits^{\infty}_{0}(\varphi'_{0})^2\left[\frac{d}{dx}\left(\frac{v_j}{\varphi'_{0}}\right)\right]^2dx+
\sum\limits_{j=1}^N\left[-v_j'v_j+v_j^2\frac{\varphi''_{0}}{\varphi'_{0}}\right]^{\infty}_{0}\\&=\sum\limits_{j=1}^N\int\limits^{\infty}_{0}(\varphi'_{0})^2\left[\frac{d}{dx}\left(\frac{v_j}{\varphi'_{0}}\right)\right]^2dx\geq 0.
\end{split}
\end{equation*}
Note that the equality 
\begin{equation*}\label{nonintegral}
\sum\limits_{j=1}^N\left[-v_j'v_j+v_j^2\frac{\varphi''_{0}}{\varphi'_{0}}\right]^{\infty}_{0}=0
\end{equation*} follows from the condition $v_j(0)=0$ and the fact that $x=0$ is the first-order zero for $\varphi'_0(x)$ (i.e. $\varphi''_0(0)\neq 0$).

Next we need to prove that $n_\pm(\bb T_0)=1$. 
 First, we  establish the  scale of Hilbert spaces associated with the self-adjoint non-negative operator  (see \cite[Section I,\S 1.2.2]{ak})
$$\bb T=\opt_0+3 \bb I,\quad \dom(\bb T)=\left\{\mathbf{V}\in W^2(\mathcal{G}):  v_1(0)=\dots=v_N(0), \sum\limits_{j=1}^N  v_j'(0)=0 \right\}.$$
Define for $s\geq 0$ the space 
$$
\mathfrak H_s(\bb T)=\left\{\bb V\in L^2(\mathcal G): \|\bb V\|_{s,2}=\Big\|(\bb T+ \bb I)^{s/2}\bb V\Big\|_2<\infty\right\}.
$$
The space $\mathfrak H_s(\bb T)$ with norm $\|\cdot\|_{s,2}$ is complete. The dual space  of $\mathfrak H_s(\bb T)$ is denoted by $\mathfrak H_{-s}(\bb T)=\mathfrak H_s(\bb T)'$. The norm in the space $\mathfrak H_{-s}(\bb T)$ is defined by the formula
$$
\|\bb V\|_{-s,2}=\Big \|(\bb T+ \bb I)^{-s/2} \bb V\Big\|_2.
$$
The spaces $\mathfrak H_s(\bb T)$ form the following chain 
\begin{equation*}\label{triples}
\dots\subset \mathfrak H_2(\bb T)\subset \mathfrak H_1(\bb T)\subset L^2(\mathcal G)=\mathfrak H_0(\bb T)\subset\mathfrak H_{-1}(\bb T)\subset \mathfrak H_{-2}(\bb T)\subset \dots
\end{equation*}
 The norm of the space 
$\mathfrak H_1(\bb T)$ can be calculated as follows
\begin{equation*}
\begin{split}
&\|\bb V\|^2_{1,2}=((\bb T+ \bb I)^{1/2}\bb V, (\bb T+\bb I)^{1/2}\bb V)\\&=\sum\limits_{j=1}^N\int\limits_0^\infty\left( |v'_j(x)|^2 +|v_j(x)|^2 +x^2|v_j(x)|^2\right)dx.
\end{split}
\end{equation*}
Therefore, we have the embedding $\mathfrak H_1(\bb T) \hookrightarrow  H^1(\mathcal G)$ and, by the Sobolev embedding, $\mathfrak H_1(\bb T) \hookrightarrow L^\infty (\mathcal G)$. From the former remark we obtain that the functional $\delta_1: \mathfrak H_1(\bb T)\to \mathbb C$ acting as  $\delta_1(\bb V)=v_1(0)$ belongs to $\mathfrak H_1(\bb T)'=\mathfrak H_{-1}(\bb T)$ and consequently  $\delta_1\in \mathfrak H_{-2}(\bb T)$. Therefore, using   \cite[Lemma 1.2.3]{ak}, it follows  that the restriction $\hat{\bb T}_0$ of the operator $\bb T$  onto the domain 
$$
\dom(\hat{\bb T}_0)=\{\bb V\in \dom(\bb T): \delta_1(\bb V)=v_1(0)=0\}=\dom(\bb T_0)$$
 
is a densely defined symmetric operator with  equal deficiency indices   $n_{\pm}(\hat{\bb T}_0)=1$. By    \cite[Chapter IV, Theorem 6]{Nai67}, the operators $\hat{\bb T}_0$ and $\bb T_0$ have equal deficiency indices. Since $\bb T_1^0$ is the self-adjoint extension of $\bb T_0$, by  Proposition \ref{semibounded},  we get $n(\bb T_1^0)\leq 1$. Using $\bb T_1^0\bb \Phi_0=-2\bb \Phi_0,$ where $\bb \Phi_0=(\varphi_0,\dots,\varphi_0)$, we arrive ate $n(\bb T_1^0)=1$. Since $\bb \Phi_0\in L^2_k(\mathcal{G})$ for any $k$, one concludes $n(\bb T_1^0|_{L^2_k(\mathcal{G})})=1$.

  $(iv)$\, See the proof of Proposition \ref{grafoN2}$(iii)$. 
\end{proof}
\begin{remark} Observe that, when we deal with the deficiency indices, the operator
$\bb T_0$
is assumed to act on complex-valued functions which however does not affect
the analysis of the negative spectrum of $\bb T_1^0$
acting on real-valued functions.
\end{remark}
Combining Lemma \ref{analici} and Theorem \ref{spect_L^0_1}, in the framework of the perturbation theory we obtain the following proposition. 
\begin{proposition}\label{perteigen} Let  $k=1,\dots,\left[\tfrac{N-1}{2}\right]$. Then there  exist $\alpha_0>0$ and two analytic functions $\mu_k : (-\alpha_0,\alpha_0)\to \mathbb R$ and $\bb E_k: (-\alpha_0,\alpha_0)\to L^2_k(\mathcal{G})$ such that
\begin{enumerate}
\item[$(i)$] $\mu_k(0)=0$ and $\bb E_k(0)=\mathbf{\widetilde{\Phi}}_{0,k}$, where $\mathbf{\widetilde{\Phi}}_{0,k}$ is defined by \eqref{Psi_0}.

\item[$(ii)$] For all $\alpha\in (-\alpha_0,\alpha_0)$, $\lambda_k(\alpha)$ is the  simple isolated second eigenvalue of $\bb T^\alpha_{1,k}$ in $L^2_k(\mathcal{G})$, and $\bb E_k(\alpha)$ is the associated eigenvector for $\lambda_k(\alpha)$.

\item[$(iii)$] $\alpha_0$ can be chosen small enough to ensure that for  $\alpha\in (-\alpha_0,\alpha_0)$  the spectrum of $\bb T_{1,k}^\alpha$ in $L^2_k(\mathcal{G})$ is positive, except at most the  first two eigenvalues.
\end{enumerate}
\end{proposition}
\begin{proof} Using the  structure of the spectrum of the operator $\bb T^0_1$  given in   Theorem \ref{spect_L^0_1}$(ii)-(iv)$, we can separate the spectrum $\sigma(\bb T^0_1)$ in $L^2_k(\mathcal{G})$  into two parts $\sigma_0=\{\mu^0_1, 0\}$, $\mu^0_1<0$, and  $\sigma_1$ by a closed curve  $\Gamma$ (for example, a circle), such that $\sigma_0$ belongs to the inner domain of $\Gamma$ and $\sigma_1$ to the outer domain of $\Gamma$ (note that $\sigma_1\subset (\epsilon, +\infty)$ for $\epsilon>0$).  Next,  Lemma \ref{analici} and the analytic perturbations theory imply  that  $\Gamma\subset \rho(\bb T^\alpha_{1, k})$ for sufficiently small $|\alpha |$, and $\sigma (\bb T^\alpha_{1, k})$ is likewise separated by $\Gamma$ into two parts, such   that the part of $\sigma (\bb T^\alpha_{1, k})$ inside $\Gamma$ consists of a finite number of eigenvalues with total multiplicity (algebraic) two. Therefore, we obtain from the Kato-Rellich Theorem (see  \cite[Theorem XII.8]{RS}) the existence of two analytic functions $\mu_k, \bb E_k$ defined in a neighborhood of zero such that  the items $(i)$, $(ii)$ and $(iii)$ hold.
\end{proof}

Now we investigate how the perturbed second eigenvalue moves depending on the sign of $\alpha$.

\begin{proposition}\label{signeigen} There exists $0<\alpha_1<\alpha_0$ such that $\lambda_k(\alpha)<0$ for any $\alpha\in (-\alpha_1,0)$, and $\lambda_k(\alpha)>0$ for  any $\alpha\in (0, \alpha_1)$. Thus,   for $\alpha$  close to $0$,  we have $n(\bb T^\alpha_{1, k}|_{L^2_k(\mathcal G)})=2$  as $\alpha<0$, and $n(\bb T^\alpha_{1, k}|_{L^2_k(\mathcal G)})=1$  as $\alpha>0$. 
\end{proposition}

\begin{proof} 
 Using analyticity of $(\bb T_{1,k}^\alpha)$, we obtain for sufficiently   small $\alpha$  the following expansions of the second eigenvalue $\mu_k$ and the corresponding eigenfunction of $\bb T_{1,k}^\alpha$ in $L^2_k(\mathcal{G})$ 
  \begin{equation}\label{decomp1_log}
\mu_k(\alpha)=\mu_{0,k} \alpha+ O(\alpha^2)\quad\text{and}\quad \bb E_k(\alpha)=\mathbf{\widetilde{\Phi}}_{0,k}+ \alpha \bb E_{0,k}  +   \bb O(\alpha^2).
\end{equation}
 From \eqref{decomp1_log} we get 
 \begin{equation}\label{1_log}
  (\bb T_{1,k}^\alpha\bb E_k(\alpha),\bb{\widetilde{\Phi}}_{0,k})=\mu_{0,k}\alpha||\bb{\widetilde{\Phi}}_{0,k}||_2^2+O(\alpha^2).
  \end{equation}
Using   
\begin{equation*}\label{T_Psi}
\begin{split}
(\bb T_{1,k}^\alpha\bb{\widetilde{\Phi}}_{0,k})_j= \left\{
                    \begin{array}{ll}
                   \tfrac{N-k}{k}\left(-\tfrac{2\alpha}{2k-N}x+\tfrac{\alpha^2}{(2k-N)^2}\right)\varphi'_0, & \,\hbox{$j=1,\dots ,k$;} \\
                   -\left(\tfrac{2\alpha}{2k-N}x+\tfrac{\alpha^2}{(2k-N)^2}\right)\varphi'_0, & \,\hbox{$j=k+1,\dots ,N,$}
                    \end{array}
                  \right.
        \end{split}
\end{equation*} 
we obtain 
\begin{equation}\label{2_log}
\begin{split}
&(\bb T_{1,k}^\alpha\bb E_k(\alpha),\bb{\widetilde{\Phi}}_{0,k})=(\bb E_k(\alpha),\bb T_{1,k}^\alpha\bb{\widetilde{\Phi}}_{0,k})=(\bb{\widetilde{\Phi}}_{0,k},\bb T_{1,k}^\alpha\bb{\widetilde{\Phi}}_{0,k})+O(\alpha^2)\\&=\tfrac{2\alpha(N-k)}{k}\int\limits_0^\infty x(\varphi'_0)^2dx+O(\alpha^2).
\end{split}
\end{equation}
Combining \eqref{1_log} and \eqref{2_log}, we get 
$$\mu_{0,k}=\frac{\tfrac{2(N-k)}{k}\int\limits_0^\infty x(\varphi'_0)^2dx}{||\widetilde{\bb \Phi}_{0,k}||^2_2}+O(\alpha).$$
It is easily seen that $\mu_{0,k}>0$ for small $\alpha$. 
Therefore, $n(\bb T_{1,k}^\alpha|_{L^2_k(\mathcal{G})})=2$ for $\alpha<0$, and  $n(\bb T_{1,k}^\alpha|_{L^2_k(\mathcal{G})})=1$ for $\alpha>0$.
\end{proof}

Now we can count the number of negative eigenvalues of $\bb T^\alpha_{1,k}$ in   $L^2_k(\mathcal{G})$ for any $\alpha$, using a classical continuation argument based on the Riesz-projection. 

\begin{proposition}\label{n(L_1)} 
Let $k=1,\dots,\left[\tfrac{N-1}{2}\right]$. Then the following assertions hold. 
\begin{enumerate}
\item[$(i)$] If  $\alpha>0$, then  $n(\bb T^\alpha_{1,k}|_{L^2_k(\mathcal{G})})=1$.
\item[$(ii)$] If $\alpha<0$, then  $n(\bb T^\alpha_{1,k}|_{L^2_k(\mathcal{G})})=2$.
\end{enumerate}
\end{proposition}

\begin{proof} We consider  the case $\alpha<0$. Recall that  $\ker(\bb T^\alpha_{1,k})=\{\bb 0\}$  by Proposition \ref{grafoN2}. Define $\alpha_\infty$ by
$$
\alpha_\infty=\inf \{\tilde \alpha<0:  \bb T^\alpha_{1,k}\;{\text{has exactly two negative eigenvalues in}\,\, L^2_k(\mathcal{G})\,\, \text{for all}}\; \alpha \in (\tilde \alpha,0)\}.
$$
 Proposition \ref{signeigen} implies that $\alpha_\infty$ is well defined and $\alpha_\infty\in [-\infty,0)$. We claim that $\alpha_\infty=-\infty$. Suppose that $\alpha_\infty> -\infty$. Let $M=n(\bb T^{\alpha_{\infty}}_{1,k}|_{L^2_k(\mathcal{G})})$ and $\Gamma$  be a closed curve (for example, a circle or a rectangle) such that $0\in \Gamma\subset \rho(\bb T^{\alpha_{\infty}}_{1,k})$, and  all the negative eigenvalues of  $\bb T^{\alpha_{\infty}}_{1,k}$ belong to the inner domain of $\Gamma$.  The existence of such $\Gamma$ can be  deduced from the lower semi-boundedness of the quadratic form associated to $\bb T^{\alpha_{\infty}}_{1,k}$.
 
Next, from Lemma \ref{analici} it  follows  that there is  $\epsilon>0$ such that for $\alpha\in [\alpha_{\infty}-\epsilon, \alpha_{\infty}+\epsilon]$ we have $\Gamma\subset \rho(\bb T^\alpha_{1,k})$ and for $\xi \in \Gamma$,
$\alpha\to (\bb T^\alpha_{1,k}-\xi)^{-1}$ is analytic. Therefore, the existence of an analytic family of Riesz-projections $\alpha\to P(\alpha)$  given by 
$$
P(\alpha)=-\frac{1}{2\pi i}\int\limits_{\Gamma} (\bb T^{\alpha}_{1,k}-\xi)^{-1}d\xi
$$
implies  that $\dim(\ran P(\alpha))=\dim(\ran P(\alpha_\infty))=M$ for all $\alpha\in [\alpha_\infty-\epsilon, \alpha_{\infty}+\epsilon]$. Next, by definition of $\alpha_\infty$, $\bb T^{\alpha_{\infty}+\epsilon}_{1, k} $ has two negative eigenvalues and $M=2$, hence, $\bb T^\alpha_{1,k}$ has two negative eigenvalues for $\alpha\in (\alpha_{\infty}-\epsilon, 0)$,  which contradicts with the definition of $\alpha_{\infty}$. Therefore,  $\alpha_{\infty}=-\infty$. 
\end{proof}
 \begin{remark}\label{Morse_est}
In  Proposition \ref{grafoN'}   we show the following estimates of $n(\bb T_{1,k}^\alpha)$ in the whole space $L^2(\mathcal{G})$:
 
 $\bullet$ \,\,$n(\bb T_{1,k}^\alpha)\leq k+1$ for $\alpha<0$;
 
 $\bullet$ \,\,$n(\bb T_{1,k}^\alpha)\leq N-k$ for $\alpha>0$.
 
 We believe that these estimates might be useful for the investigation of the orbital instability of  the standing  waves $e^{i\omega t}\mathbf{\Phi}_{k}^\alpha$ in $W_{\EE}^1(\mathcal{G})$.
  \end{remark}
\noindent{\bf Proof of Theorem \ref{main_log}.}

Due to Theorem \ref{7main}, we have  $n(\bb H_{k}^\alpha|_{L^2_k(\mathcal{G})})=2$ for $\alpha<0$,  and  $n(\bb H_{k}^\alpha|_{L^2_k(\mathcal{G})})=1$ for $\alpha>0$.
Using  Theorem \ref{stabil_graph_log}, we obtain orbital stability of $e^{i\omega t}\bb \Phi_k$ in $W^1_{\EE,k}(\mathcal{G})$ for $\alpha>0$ and  spectral instability for $\alpha<0$. 
 \section{Appendix}\label{app}

Below  we  show the estimates  from Remark \ref{Morse_est}. We use the following  abstract result (see \cite{Nai67}).
\begin{proposition}\label{semibounded}
Let $A$  be a densely defined lower semi-bounded symmetric operator (that is, $A\geq mI$)  with finite deficiency indices $n_{\pm}(A)=n<\infty$  in the Hilbert space $\mathcal{H}$, and let $\widetilde{A}$ be a self-adjoint extension of $A$.  Then the spectrum of $\widetilde{A}$  in $(-\infty, m)$ is discrete and  consists of at most $n$  eigenvalues counting multiplicities.
\end{proposition}
 \begin{proposition}\label{grafoN'}
Let $\alpha\neq 0$,   $k=1,\dots,\left[\tfrac{N-1}{2}\right]$.   Then the following assertions hold. 
\begin{itemize}
\item[$(i)$] If  $\alpha< 0$, then $n(\mathbf{T}^\alpha_{1,k})\leq k+1$.
\item[$(ii)$] If $\alpha>0$, then $n(\mathbf{T}^\alpha_{1,k})\leq N-k$.
  \end{itemize}
\end{proposition}
\begin{proof}  $(i)$ 
First,  note that $\mathbf{T}^\alpha_{1,k}$ is the self-adjoint extension of the  symmetric operator 
\begin{equation*}
\begin{split}
&\mathbf{T}_{0,k}=\left(\Big(-\frac{d^2}{dx^2}+(x\pm \tfrac{\alpha}{2k-N})^2-3\Big)\delta_{i,j}\right),\\
&\dom(\mathbf{T}_{0,k})=\left\{\begin{array}{c}
\bb V\in W^2(\mathcal{G}):  v_1(0)=\dots=v_N(0)=0,\\
\sum\limits_{j=1}^N  v_j'(0)=0,\, v_1(a_k)=\dots=v_k(a_k)=0 
\end{array}\right\}, \,\, a_k=\frac{\alpha}{2k-N}.
\end{split}
\end{equation*}
Below we show that the operator $\mathbf{T}_{0,k}$ is non-negative and  $n_\pm(\mathbf{T}_{0,k})=k+1$ (when $\bb T_{0,k}$  is assumed to act  on complex-valued  functions).

 First, note that every component of the vector $\mathbf{V}=(v_j)_{j=1}^N\in W^2(\mathcal{G})$ satisfies the following identity 
\begin{equation}\label{identity_graph_k}  
-v_j''+\omega v_j+\left((x\pm \tfrac{\alpha}{2k-N})^2-3\right)v_j=
\frac{-1}{\varphi_{k,j}'}\frac{d}{dx}\left[(\varphi_{k,j}')^2\frac{d}{dx}\left(\frac{v_j}{\varphi_{k,j}'}\right)\right],\quad x\in\mathbb{R}_+\setminus \{a_k\}.
\end{equation}
Moreover, for $j\in\{k+1,\dots,N\}$ the above equality holds also for $a_k$ since $\varphi_{k,j}'(a_k)\neq 0,\, j\in\{k+1,\dots,N\}$. 
Using the above equality and integrating by parts, we get for $\mathbf{V}\in \dom(\mathbf{T}_{0,k})$
 \begin{equation*}\label{nonneg_graph}\begin{split}
&(\mathbf{T}_{0,k}\mathbf{V},\mathbf{V})=\sum\limits_{j=1}^k \Big(\int\limits_{0}^{a_k-} + \int\limits_{a_k+}^{+\infty}\Big) 
(\varphi_{k,j} ')^2\left[\frac{d}{dx}\left(\frac{v_j}{\varphi_{k,j} '}\right)\right]^2dx+
\sum\limits_{j=k+1}^N\int\limits^{\infty}_{0}(\varphi_{k,j}')^2\left[\frac{d}{dx}\left(\frac{v_j}{\varphi_{k,j} '}\right)\right]^2dx\\&+
\sum\limits_{j=1}^N\left[-v_j'{v}_j+v_j^2\frac{\varphi_{k,j}''}{\varphi_{k,j}'}\right]^{\infty}_{0}+\sum\limits_{j=1}^k\left[v_j'{v}_j-v_j^2\frac{\varphi_{k,j}''}{\varphi_{k,j}'}\right]^{a_k+}_{a_k-}\\&=\sum\limits_{j=1}^k \Big( \int\limits_{0}^{a_k-} + \int\limits_{a_k+}^{+\infty}\Big) 
(\varphi_{k,j} ')^2\left[\frac{d}{dx}\left(\frac{v_j}{\varphi_{k,j} '}\right)\right]^2dx+
\sum\limits_{j=k+1}^N\int\limits^{\infty}_{0}(\varphi_{k,j}')^2\left[\frac{d}{dx}\left(\frac{v_j}{\varphi_{k,j} '}\right)\right]^2dx\geq 0.
\end{split}
\end{equation*}

The  equality $\sum\limits_{j=1}^k\left[v_j'{v}_j-v_j^2\frac{\varphi_{k,j}''}{\varphi_{k,j}'}\right]^{a_k+}_{a_k-}=0$ needs an additional explanation. Indeed, since $a_k$ is a zero of the first order for $\varphi_{k,j}'$ (i.e. $\varphi_{k,j}''(a_k)\neq 0$),  $v_j\in H^2(\mathbb{R_+})$ and  $v_j(a_k)=0$, we get  $
\lim\limits_{x\rightarrow a_k}\displaystyle{\frac{v^2_j(x)}{\varphi_{k,j}'(x)}}=\lim\limits_{x\rightarrow a_k}\displaystyle{\frac{2v_j(x)v'_j(x)}{\varphi_{k,j}''(x)}}=0.$
To prove $n_\pm(\mathbf{T}_{0,k})=k+1$ we will use the idea of the proof of Theorem \ref{spect_L^0_1}$(iii)$.
Consider the following non-negative self-adjoint operator 
\begin{equation*}
\begin{split}
&\bb T_k=\left(\Big(-\frac{d^2}{dx^2}+(x\pm \tfrac{\alpha}{2k-N})^2\Big)\delta_{i,j}\right),\\
& \dom(\bb T_k)=\left\{\mathbf{V}\in W^2(\mathcal{G}):  v_1(0)=\dots=v_N(0), \sum\limits_{j=1}^N  v_j'(0)=0 \right\}.
\end{split}
\end{equation*}
As in the proof of Theorem \ref{spect_L^0_1}$(iii)$, we define chain of the Hilbert spaces
\begin{equation*}\label{triples}
\dots\subset \mathfrak H_2(\bb T_k)\subset \mathfrak H_1(\bb T_k)\subset L^2(\mathcal G)=\mathfrak H_0(\bb T_k)\subset\mathfrak H_{-1}(\bb T_k)\subset \mathfrak H_{-2}(\bb T_k)\subset\dots
\end{equation*}
We have the embedding $\mathfrak H_1(\bb T_k) \hookrightarrow  H^1(\mathcal G)$ and, by the Sobolev embedding, $\mathfrak H_1(\bb T_k) \hookrightarrow L^\infty (\mathcal G)$. From the former remark we obtain that the functionals
\begin{equation*}
\begin{split}
&\delta_1: \mathfrak H_1(\bb T_k)\to \mathbb C,\quad \delta_{j, a_k}:  \mathfrak H_1(\bb T_k)\to \mathbb C,\\
&\delta_1(\bb V)=v_1(0),\quad  \delta_{j,a_k}(\bb V)=v_j(a_k),\,\,j\in\{1,\dots,k\},
\end{split}
\end{equation*}  belong to $\mathfrak H_1(\bb T_k)'=\mathfrak H_{-1}(\bb T_k)$ and consequently  $\delta_1,\delta_{j, a_k}\in \mathfrak H_{-2}(\bb T_k)$. Therefore, using   \cite[Lemma 3.1.1]{ak}, it follows  that the restriction $\hat{\bb T}_{0,k}$ of the operator $\bb T_k$  onto the domain 
$$
\dom(\hat{\bb T}_{0,k})=\left\{\begin{array}{c}
\bb V\in \dom(\bb T_k): \delta_1(\bb V)=v_1(0)=0,\\
\delta_{j,a_k}(\bb V)=v_j(a_k)=0,\,j\in\{1,\dots,k\}\end{array}\right\}=\dom(\bb T_{0,k})$$
 is a densely defined symmetric operator with  equal deficiency indices   $n_{\pm}(\hat{\bb T}_{0,k})=k+1$. By    \cite[Chapter IV, Theorem 6]{Nai67}, the operators $\hat{\bb T}_{0,k}$ and $\bb T_{0,k}$ have equal deficiency indices. Therefore, $n(\mathbf{T}^\alpha_{1,k})\leq k+1$.

$(ii)$ The proof is similar. In particular, we need to consider the operator $\bb T^\alpha_{1,k}$ as  the self-adjoint extension of the non-negative symmetric operator \begin{equation*}
\begin{split}
&\mathbf{T}_{0,N-k}=\left(\Big(-\frac{d^2}{dx^2}+(x\pm \tfrac{\alpha}{2k-N})^2-3\Big)\delta_{i,j}\right),\\ &\dom(\mathbf{T}_{0,N-k})=\left\{
\bb V\in \dom(\bb T^\alpha_{1,k}): v_{k+1}(a_k)=\dots=v_N(a_k)=0 
\right\}. 
\end{split}
\end{equation*}
The deficiency indices of $\mathbf{T}_{0,N-k}$ equal $N-k$ (since basically $\mathbf{T}_{0,N-k}$ is the restriction of the operator $\bb T^\alpha_{1,k}$ onto the subspace of  codimension $N-k$). To show the non-negativity of $\mathbf{T}_{0,N-k}$,  we need to use formula \eqref{identity_graph_k}. It induces 
\begin{equation*}\label{nonneg_graph}\begin{split}
&(\mathbf{T}_{0,N-k}\mathbf{V},\mathbf{V})=\sum\limits_{j=k+1}^N \Big( \int\limits_{0}^{a_k-}+\int\limits_{a_k+}^{+\infty}\Big) 
(\varphi_{k,j} ')^2\left[\frac{d}{dx}\left(\frac{v_j}{\varphi_{k,j} '}\right)\right]^2dx+
\sum\limits_{j=1}^k\int\limits^{\infty}_{0}(\varphi_{k,j}')^2\left[\frac{d}{dx}\left(\frac{v_j}{\varphi_{k,j} '}\right)\right]^2dx  \\&+
\sum\limits_{j=1}^N\left[-v_j'{v}_j+v_j^2\frac{\varphi_{k,j}''}{\varphi_{k,j}'}\right]^{\infty}_{0}+\sum\limits_{j=k+1}^N\left[v_j'{v}_j-v_j^2\frac{\varphi_{k,j}''}{\varphi_{k,j}'}\right]^{a_k+}_{a_k-}\\&=\sum\limits_{j=k+1}^N \Big( \int\limits_{0}^{a_k-} + \int\limits_{a_k+}^{+\infty}\Big) 
(\varphi_{k,j} ')^2\left[\frac{d}{dx}\left(\frac{v_j}{\varphi_{k,j} '}\right)\right]^2dx+
\sum\limits_{j=1}^k\int\limits^{\infty}_{0}(\varphi_{k,j}')^2\left[\frac{d}{dx}\left(\frac{v_j}{\varphi_{k,j} '}\right)\right]^2dx\\&+
\sum\limits_{j=1}^N\left[v_j'(0){v}_j(0)-v^2_j(0)\frac{\varphi_{k,j}''(0)}{\varphi_{k,j}'(0)}\right]\geq 0.
\end{split}
\end{equation*}
Indeed, $\sum\limits_{j=k+1}^N\left[v_j'{v}_j-v_j^2\frac{\varphi_{k,j}''}{\varphi_{k,j}'}\right]^{a_k+}_{a_k-}=0$ (see the proof of item $(i)$). Moreover,
 $$\sum\limits_{j=1}^N\left[v_j'(0){v}_j(0)-v_j^2(0)\frac{\varphi_{k,j}''(0)}{\varphi_{k,j}'(0)}\right]=v_1^2(0)\frac{(N-2k)^2}{\alpha}\geq 0.$$ 
 Finally,  due to Proposition \ref{semibounded}, we get  the result.
\end{proof}
\begin{remark}
\begin{itemize}
\item[$(i)$] It is easily seen that 
 $$\sum\limits_{j=1}^N\left[v_j'(0){v}_j(0)-v_j^2(0)\frac{\varphi_{k,j}''(0)}{\varphi_{k,j}'(0)}\right]=v_1^2(0)\frac{(N-2k)^2}{\alpha}\leq 0$$
  for $\alpha<0$, and, therefore, the  restriction  of  $\bb T^\alpha_{1,k}$  onto the subspace 
  $$\{\bb V\in W^2(\mathcal{G}): v_1(a_k)=\dots=v_k(a_k)=0\}$$ is not  a non-negative operator as $\alpha<0$. Thus, we need  to assume additionally that $v_1(0)=\dots=v_N(0)=0$. 
  \item[$(ii)$]  The result of the item $(ii)$ (for $\alpha>0$) of the above Proposition can be extended to the case of $k=0$, i.e. $n(\bb T_{1,0}^\alpha)\leq N$. 
 
  \end{itemize}
\end{remark}

\section*{Acknowledgements.}

The author  was supported by FAPESP under the project 2016/02060-9.

\end{document}